\documentclass[reqno]{amsart}
\usepackage{hyperref}

\usepackage{amsmath,amssymb,amsthm}
\usepackage{enumerate}
\usepackage{color}
\usepackage{epsfig}
\usepackage{graphics}
\usepackage{graphicx}
\usepackage{subfigure}

\begin{document}
\title[\hfilneg  \hfil ]{Quasi-parabolic Composition Operators on Weighted Bergman Spaces}

\author[u. g\"{u}l]{u\u{g}ur g\"{u}l}

\address{u\u{g}ur g\"{u}l,  \newline
Hacettepe University, Department of Mathematics, 06800, Beytepe,
Ankara, TURKEY}
\email{\href{mailto:gulugur@gmail.com}{gulugur@gmail.com}}


\thanks{Submitted 18 April 2015}

\subjclass[2000]{47B33} \keywords{Composition Operators, Bergman
Spaces, Essential Spectra.}

\begin{abstract}
    In this work we study the essential spectra of composition operators
    on weighted Bergman spaces of analytic functions which might be termed as
    ``quasi-parabolic.'' This is the class of composition operators on
    $A_{\alpha}^{2}$ with symbols whose conjugate with the Cayley transform on
    the upper half-plane are of the form $\varphi(z)=$ $z+\psi(z)$, where
    $\psi\in$ $H^{\infty}(\mathbb{H})$ and $\Im(\psi(z)) > \epsilon > 0$.
    We especially examine the case where $\psi$ is discontinuous at
    infinity. A new method is devised to show that this type of
    composition operators fall in a C*-algebra of Toeplitz operators and
    Fourier multipliers. This method enables us to provide new examples
    of essentially normal composition operators and to calculate their
    essential spectra.
\end{abstract}

\maketitle
\newtheorem{theorem}{Theorem}
\newtheorem{acknowledgement}[theorem]{Acknowledgement}
\newtheorem{algorithm}[theorem]{Algorithm}
\newtheorem{axiom}[theorem]{Axiom}
\newtheorem{case}[theorem]{Case}
\newtheorem{claim}[theorem]{Claim}
\newtheorem{conclusion}[theorem]{Conclusion}
\newtheorem{condition}[theorem]{Condition}
\newtheorem{conjecture}[theorem]{Conjecture}
\newtheorem{corollary}[theorem]{Corollary}
\newtheorem{criterion}[theorem]{Criterion}
\newtheorem{definition}[theorem]{Definition}
\newtheorem{example}[theorem]{Example}
\newtheorem{exercise}[theorem]{Exercise}
\newtheorem{lemma}[theorem]{Lemma}
\newtheorem{notation}[theorem]{Notation}
\newtheorem{problem}[theorem]{Problem}
\newtheorem{proposition}[theorem]{Proposition}
\newtheorem{remark}[theorem]{Remark}
\newtheorem{solution}[theorem]{Solution}
\newtheorem{summary}[theorem]{Summary}
\newtheorem*{thma}{Theorem A}
\newtheorem*{thmb}{Theorem B}
\newtheorem*{thmc}{Theorem C}
\newtheorem*{thmd}{Theorem D}
\newtheorem*{Paley-WienerHardy}{Paley-Wiener Theorem}
\newtheorem*{Paley-Wiener}{Paley-Wiener Theorem for Weighted Bergman Spaces}
\newtheorem*{Power}{Power's Theorem}
\newtheorem*{mthm1}{Main Theorem 1}
\newtheorem*{mthm2}{Main Theorem 2}
\newcommand{\norm}[1]{\left\Vert#1\right\Vert}
\newcommand{\abs}[1]{\left\vert#1\right\vert}
\newcommand{\set}[1]{\left\{#1\right\}}
\newcommand{\Real}{\mathbb R}
\newcommand{\eps}{\varepsilon}
\newcommand{\To}{\longrightarrow}
\newcommand{\BX}{\mathbf{B}(X)}
\newcommand{\A}{\mathcal{A}}

\section{introduction}
This paper is a continuation of our work \cite{Gul} on
quasi-parabolic composition operators on the Hardy space $H^{2}$.
In this work we investigate the same class of operators on
weighted Bergman spaces $A^{2}_{\alpha}(\mathbb{D})$.

Quasi-parabolic composition operators is a generalization of the
composition operators induced by parabolic linear fractional
non-automorphisms of the unit disc that fix a point $\xi$ on the
boundary. These linear fractional transformations for $\xi= 1$
take
  the form
  \[\varphi_{a}(z)=\frac{2iz+a(1-z)}{2i+a(1-z)}\]
  with $\Im(a) > 0$. Quasi-parabolic composition operators on $H^{2}(\mathbb{D})$
  are composition operators induced by the symbols where `$a$' is replaced by a bounded
  analytic function `$\psi$' for which $\Im(\psi(z))>\delta>0$\quad
  $\forall z\in\mathbb{D}$. We recall that the local essential range $\mathcal{R}_{\xi}(\psi^{\ast})$ of $\psi\in
  H^{\infty}(\mathbb{D})$ at $\xi\in\mathbb{T}$ is defined to be the set of points
  $\zeta\in$ $\mathbb{C}$ for which the set $\{z\in\mathbb{T}:\mid\psi^{\ast}(z)-\zeta\mid<\varepsilon\}\cap
  S_{\xi,r}$ has positive Lebesgue measure $\forall\varepsilon>0$ and $\forall r>0$ where $S_{\xi,r}=\{z\in\mathbb{T}:\mid z-\xi\mid<r\}$ and
  $\psi^{\ast}\in L^{\infty}(\mathbb{T})$ is the boundary value function of $\psi$. In \cite{Gul} we showed that if $\psi\in
  QC(\mathbb{T})\cap H^{\infty}(\mathbb{D})$ then these composition operators are essentially
  normal and their essential spectra are given as
  \[\sigma_{e}(C_{\varphi})=\{e^{izt}:t\in
  [0,\infty],z\in\mathcal{R}_{1}(\psi^{\ast})\}\cup\{0\}\]
where\quad $\mathcal{R}_{1}(\psi^{\ast})$ is the local essential
range of $\psi$ at $1$.

In the weighted Bergman space setting $QC$ is replaced by
$VMO_{\partial}$. The class of ``Vanishing Mean Oscillation near
the Boundary" functions is defined as the set of functions $f\in
 L^{1}(\mathbb{D},dA)$ satisfying
 $$\lim_{\mid z\mid\rightarrow 1^{-}}\frac{1}{\mid Q_{z}\mid}\int_{Q_{z}}\mid
 f(w)-\frac{1}{\mid Q_{z}\mid}\int_{Q_{z}}f(u)dA(u)\mid dA(w)=0$$
 where $Q_{z}=\{w\in\mathbb{D}:\mid w\mid\geq\mid z\mid,\mid\arg
 w-\arg z\mid\leq 1-\mid z\mid\}$ and $\mid Q_{z}\mid=(1+\mid
 z\mid)(1-\mid z\mid)^{2}$ is the $dA$ measure of $Q_{z}$. We have the following very
similar result in the weighted Bergman space setting:
\begin{mthm1}
Let $\varphi:\mathbb{D}\rightarrow$ $\mathbb{D}$ be an analytic
self-map of $\mathbb{D}$ such that
\begin{equation*}
\varphi(z)=\frac{2iz+\eta(z)(1-z)}{2i+\eta(z)(1-z)}
\end{equation*}
where $\eta\in$ $H^{\infty}(\mathbb{D})$ with $\Im(\eta(z)) >
\epsilon> 0$ for all $z\in$ $\mathbb{D}$. If $\eta\in$
$VMO_{\partial}(\mathbb{D})\cap H^{\infty}$ then we have
\begin{itemize}
\item (i)\quad $C_{\varphi}:$ $A^{2}_{\alpha}(\mathbb{D})\rightarrow$
$A^{2}_{\alpha}(\mathbb{D})$ is essentially normal
\item (ii)\quad$\sigma_{e}(C_{\varphi})=\{e^{izt}:t\in [0,\infty],
z\in\mathcal{R}_{1}(\eta^{\ast})\}\cup\{0\}$
\end{itemize}
where\quad $\mathcal{R}_{1}(\eta^{\ast})$ is the local essential
range of $\eta^{\ast}\in L^{\infty}(\mathbb{T})$ at $1$ and
$\eta^{\ast}$ is the boundary limit value function of $\eta$.
\end{mthm1}
In the upper half-plane for $\psi\in VMO_{\partial}\cap
H^{\infty}(\mathbb{H})$ the local essential range
$\mathcal{R}_{\infty}(\psi)$ of $\psi$ at $\infty$ is defined to
be the set of points $z\in$ $\mathbb{C}$ so that, for all
$\varepsilon > 0$ and $n > 0$, we have
\begin{equation*}
    \lambda((\psi^{\ast})^{-1}(B(z,\varepsilon))\cap(\mathbb{R}-[-n,n]))>0 ,
\end{equation*}
where $\lambda$ is the Lebesgue measure on $\mathbb{R}$ and
$\psi^{\ast}$ is the boundary value function of $\psi$. We have
the following result for the upper half-plane case:
\begin{mthm2}
Let $\psi\in$ $VMO_{\partial}(\mathbb{H})\cap
H^{\infty}(\mathbb{H})$ such that $\Im(\psi(z)) > \epsilon > 0$
for all $z\in$ $\mathbb{H}$ then for $\varphi(z)=$ $z+\psi(z)$ and
$\alpha>-1$ we have
\begin{itemize}
\item (i)\quad $C_{\varphi}:$ $A^{2}_{\alpha}(\mathbb{H})\rightarrow$
$A^{2}_{\alpha}(\mathbb{H})$ is essentially normal \\
\item (ii)\quad$\sigma_{e}(C_{\varphi})=\{e^{izt}:t\in [0,\infty],
z\in\mathcal{R}_{\infty}(\psi^{\ast})\}\cup\{0\}$
\end{itemize}
where\quad $\mathcal{R}_{\infty}(\psi^{\ast})$ is the local
essential range of $\psi^{\ast}\in L^{\infty}(\mathbb{R})$ at
$\infty$ and $\psi^{\ast}$ is the boundary limit value $\psi$.
\end{mthm2}

\section{Notation and Preliminaries}

In this section we fix the notation that we will use throughout
and recall some preliminary facts that will be used in the sequel.

Let $S$ be a compact Hausdorff topological space. The space of all
complex valued continuous functions on $S$ will be denoted by
$C(S)$. For any $f\in C(S)$, $\parallel f\parallel_{\infty}$ will
denote the sup-norm of $f$, i.e. $$\parallel
f\parallel_{\infty}=\sup\{\mid f(s)\mid:s\in S\}.$$ For a Banach
space $X$, $K(X)$ will denote the space of all compact operators
on $X$ and $\mathcal{B}(X)$  will denote the space of all bounded
linear operators on $X$. The open unit disc will be denoted by
$\mathbb{D}$, the open upper half-plane will be denoted by
$\mathbb{H}$, the real line will be denoted by $\mathbb{R}$ and
the complex plane will be denoted by $\mathbb{C}$. The one point
compactification of $\mathbb{R}$ will be denoted by
$\dot{\mathbb{R}}$ which is homeomorphic to $\mathbb{T}$. For any
$z\in$ $\mathbb{C}$, $\Re(z)$ will denote the real part, and
$\Im(z)$ will denote the imaginary part of $z$, respectively. For
any subset $S\subset$ $\mathcal{B}(H)$, where $H$ is a Hilbert
space, the C*-algebra generated by $S$ will be denoted by
$C^{*}(S)$. The Cayley transform $\mathfrak{C}$ will be defined by
\begin{equation*}
\mathfrak{C}(z)=\frac{z-i}{z+i}.
\end{equation*}
 For any $a\in$ $L^{\infty}(\mathbb{H})$ (or $a\in$
$L^{\infty}(\mathbb{D})$), $M_{a}$ will be the multiplication
operator on $L^{2}(\mathbb{H})$ (or $L^{2}(\mathbb{D})$) defined
as
\begin{equation*}
M_{a}(f)(x)=a(x)f(x).
\end{equation*}
For convenience, we remind the reader of the rudiments of Gelfand
theory of commutative Banach algebras and Toeplitz operators.

Let $A$ be a commutative Banach algebra. Then its maximal ideal
space $M(A)$ is defined as
\begin{equation*}
    M(A)=\{x\in A^{*}:x(ab)=x(a)x(b)\quad\forall a,b\in A\}
\end{equation*}
where $A^{*}$ is the dual space of $A$. If $A$ has identity then
$M(A)$ is a compact Hausdorff topological space with the weak*
topology. The Gelfand transform $\Gamma:A\rightarrow C(M(A))$ is
defined as
\begin{equation*}
    \Gamma(a)(x)=x(a).
\end{equation*}
 If $A$ is a commutative C*-algebra with
identity, then $\Gamma$ is an isometric *-isomorphism between $A$
and $C(M(A))$. If $A$ is a C*-algebra and $I$ is a two-sided
closed ideal of $A$, then the quotient algebra $A/I$ is also a
C*-algebra (see \cite{Arv} and \cite{Rud}).
 For $a\in A$ the spectrum $\sigma_{A}(a)$ of $a$ on $A$
is defined as
\begin{equation*}
    \sigma_{A}(a)=\{\lambda\in\mathbb{C}:\lambda e-a\ \ \textrm{is not invertible in}\ A\},
\end{equation*}
where $e$ is the identity of $A$. We will use the spectral
permanency property of C*-algebras (see \cite{Rud}, pp. 283); i.e.
if $A$ is a C*-algebra with identity and $B$ is a closed
*-subalgebra of $A$, then for any $b\in B$ we have
\begin{equation}
\sigma_{B}(b)=\sigma_{A}(b).
\end{equation}
To compute essential spectra we employ the following important
fact (see \cite{Rud}, pp. 268): If $A$ is a commutative Banach
algebra with identity then for any $a\in A$ we have
\begin{equation}
    \sigma_{A}(a)=\{\Gamma(a)(x)=x(a):x\in M(A)\}.
\end{equation}
In general (for $A$ not necessarily commutative), we have
\begin{equation}
    \sigma_{A}(a)\supseteq\{x(a):x\in M(A)\}.
\end{equation}

For a Banach algebra $A$, we denote by $com(A)$ the closed ideal
in $A$ generated by the commutators
$\{a_{1}a_{2}-a_{2}a_{1}:a_{1},a_{2}\in A\}$. It is an algebraic
fact that the quotient algebra $A/com(A)$ is a commutative Banach
algebra. The reader can find detailed information about Banach and
C*-algebras in \cite{Rud} related to what we have reviewed so far.

The essential spectrum $\sigma_{e}(T)$ of an operator $T$ acting
on a Banach
  space $X$ is the spectrum of the coset of $T$ in the Calkin algebra
  $\mathcal{B}(X)/K(X)$, the algebra of bounded linear operators modulo
  compact operators. The well known Atkinson's theorem identifies the essential
  spectrum of $T$ as the set of all $\lambda\in$ $\mathbb{C}$ for
  which $\lambda I-T$ is not a Fredholm operator. The essential norm of $T$ will be denoted by $\parallel T\parallel_{e}$ which is defined as
\begin{equation*}
 \parallel T\parallel_{e}=\inf\{\parallel T+K\parallel:K\in K(X)\}
\end{equation*}
   The bracket $[\cdot]$ will denote the equivalence class modulo
  $K(X)$. An operator $T\in\mathcal{B}(H)$ is called essentially
  normal if $T^{*}T-TT^{*}\in K(H)$ where $H$ is a Hilbert space and
  $T^{*}$ denotes the Hilbert space adjoint of $T$.

For $\alpha$ $>-1$ the weighted Bergman space
 $A_{\alpha}^{2}(\mathbb{H})$ of the upper half-plane is defined as
 $$A_{\alpha}^{2}(\mathbb{H}) = \{
 f:\mathbb{H}\rightarrow\mathbb{C}:\textrm{f is analytic and}\qquad\int_{\mathbb{H}}\mid
 f(x+iy)\mid^{2}y^{\alpha}dxdy <\infty \}$$
The weighted Bergman spaces $A^{2}_{\alpha}$ are reproducing
kernel Hilbert
 spaces with kernel functions $$k_{w}(z)= \frac{1}{(\bar{w}-z)^{\alpha +2}}$$
 (see \cite{Du-Sch}).For $H^{2}(\mathbb{D})$, the Hardy space of the unit
 disc it is quite an obvious fact that if $f\in$
 $L^{2}(\mathbb{T})$,\quad $f(z)$ $= \sum^{\infty}_{-\infty}
 \hat{f}(n)z^{n}$ then $$ f\in H^{2} \Leftrightarrow \hat{f}(n) =
 0\qquad \forall n < 0 $$ i.e. $f\in$ $H^{2}$ if and only if its
 negative Fourier coefficients are zero.

 A similar fact arises for $H^{2}(\mathbb{H})$ as the Paley-Wiener
 theorem:\\
\begin{Paley-WienerHardy}
  T.F.A.E\\
a. $F\in$ $H^{2}(\mathbb{H})$\\
b. $\exists f\in$ $L^{2}(\mathbb{R}^{+})$ s.t. $F(z)=$
 $\int_{0}^{\infty}f(t)e^{2\pi itz}dt$\quad $z\in$ $\mathbb{H}$
\end{Paley-WienerHardy}
Moreover the correspondence $F$ $\rightarrow$ $f$ is an isometric
 isomorphism of $H^{2}(\mathbb{H})$ onto $L^{2}(\mathbb{R}^{+})$.
 For weighted Bergman spaces $A^{2}_{\alpha}(\mathbb{H})$ the Paley
 Wiener theorem as proved by P. Duren, E. Gallardo-Gutierrez and A.
 Montes-Rodriguez(see \cite{Du-GG-MR}) takes the following form:
\begin{Paley-Wiener}
T.F.A.E\\
(a). $F\in$ $A^{2}_{\alpha}(\mathbb{H})$\\
(b). $\exists f\in$ $L^{2}_{\alpha +1}(\mathbb{R}^{+})$ s.t.
$F(z)=$ $\int_{0}^{\infty}f(t)e^{2\pi itz}dt$\quad $z\in$
$\mathbb{H}$\\
where
$$L^{2}_{\beta}(\mathbb{R}^{+})=\{f:\mathbb{R}^{+}\rightarrow\mathbb{C}:\int_{0}^{\infty}\mid
 f(t)\mid^{2}t^{-\beta}dt<\infty\}$$ The weighted space
 $L^{2}_{\beta}(\mathbb{R}^{+})$ is equipped with the norm
 $$(\parallel f\parallel_{L^{2}_{\beta}})^{2} =
 \frac{\Gamma(\beta)}{2^{\beta}}\int_{0}^{\infty}\mid
 f(t)\mid^{2}t^{-\beta}dt$$ The correspondence $F$ $\rightarrow$ $f$
 is again an isometric isomorphism of $A^{2}_{\alpha}(\mathbb{H})$
 onto $L^{2}_{\beta}(\mathbb{R}^{+})$. It is easily seen that this
 correspondence is the well-known Fourier transform.
\end{Paley-Wiener}

Since the map $\mathcal{F}^{-1}:L^{2}_{\alpha
+1}(\mathbb{R}^{+})\rightarrow A^{2}_{\alpha}(\mathbb{H})$,
$$\mathcal{F}^{-1}(f)(z)=\int_{0}^{\infty}f(t)e^{2\pi itz}dt$$ is
an isometric isomorphism of Hilbert spaces, it is unitary and
hence we have
\begin{eqnarray*}
& &\langle\mathcal{F}^{-1}(f),g\rangle_{A^{2}_{\alpha}(\mathbb{H})}=\int_{\mathbb{H}}\mathcal{F}^{-1}(f)(z)\overline{g(z)}dA_{\alpha}(z)=\int_{\mathbb{H}}(\int_{0}^{\infty}f(t)e^{2\pi itz}dt)\overline{g(z)}dA_{\alpha}(z)=\\
& &\langle f,\mathcal{F}(g)\rangle_{L^{2}_{\alpha
+1}(\mathbb{R}^{+})}.
\end{eqnarray*}
Using Fubini-Tonelli theorem we have
\begin{eqnarray*}
& &\int_{\mathbb{H}}(\int_{0}^{\infty}f(t)e^{2\pi itz}dt)\overline{g(z)}dA_{\alpha}(z)=\int_{0}^{\infty}f(t)(\int_{\mathbb{H}}t^{1+\alpha}e^{2\pi itz}\overline{g(z)}dA_{\alpha}(z))\frac{dt}{t^{1+\alpha}}\\
& &=\langle f,\mathcal{F}(g)\rangle_{L^{2}_{\alpha +1}(\mathbb{R}^{+})}=\frac{\Gamma(\alpha +1)}{2^{\alpha +1}}\int_{0}^{\infty}f(t)\overline{\mathcal{F}(g)(t)}\frac{dt}{t^{1+\alpha}}\\
\end{eqnarray*}
Hence we have the representation of
$\mathcal{F}:A^{2}_{\alpha}(\mathbb{H})\rightarrow L^{2}_{\alpha
+1}(\mathbb{R}^{+})$ as follows:
$$\mathcal{F}(g)(t)=\frac{2^{\alpha +1}t^{1+\alpha}}{\Gamma(\alpha +1)}\int_{\mathbb{H}}e^{-2\pi it\bar{z}}g(z)dA_{\alpha}(z)$$

 By the help of this fact one can distinguish a class of C*
algebras
 of operators acting on $A^{2}_{\alpha}(\mathbb{H})$. For $X$ being a C* algebra of
 functions of $\mathbb{R}^{+}$ s.t. $X\subseteq$
 $L^{\infty}(\mathbb{R}^{+})$ the Fourier multiplier algebra on $A^{2}_{\alpha}$ associated to $X$ is
 defined to be $$ F_{X}^{A^{2}_{\alpha}} = \{T:A^{2}_{\alpha}\rightarrow
A^{2}_{\alpha}:\mathcal{F}T\mathcal{F}^{-1} = M_{\phi}\quad\phi\in
X\}$$
 where $M_{\phi}f(x)=$ $\phi(x)f(x)$ is the multiplication
 operator by $\phi$ and $\mathcal{F}$ is the Fourier transform. Throughout we will be concerned with the case $X=$
 $C_{0}([0,\infty))$ the C* algebra of continuous functions
 vanishing at infinity. An example of a Fourier multiplier is the
 translation by $\gamma\in$ $\mathbb{H}$ i.e.  $T:A^{2}_{\alpha}\rightarrow$ $A^{2}_{\alpha}$, $Tf(z)=$
 $f(z+\gamma)$. One easily sees that $\mathcal{F}T\mathcal{F}^{-1}=$
 $M_{\phi}$ where $\phi(t)=$ $e^{2\pi i\gamma t}$. For any $\phi\in
 C([0,\infty])$ $D_{\phi}:A^{2}_{\alpha}(\mathbb{H})\rightarrow
 A^{2}_{\alpha}(\mathbb{H})$ will denote the Fourier multiplier
 defined as $$D_{\phi}=\mathcal{F}^{-1}M_{\phi}\mathcal{F}$$

 For any $f\in L^{\infty}(\mathbb{D})$, the Toeplitz
 operator $T_{f}:A^{2}_{\alpha}(\mathbb{D})\rightarrow A^{2}_{\alpha}(\mathbb{D})$ with
 symbol $f$ is defined as
 $$T_{f}=P M_{f} $$ where $P:L^{2}(\mathbb{D},dA_{\alpha})\rightarrow
 A^{2}_{\alpha}(\mathbb{D})$ is the orthogonal projection and $M_{f}:L^{2}(\mathbb{D},dA_{\alpha})\rightarrow
 L^{2}(\mathbb{D},dA_{\alpha})$ is the multiplication operator.
 Similarly the Toeplitz operator $T_{f}:A^{2}_{\alpha}(\mathbb{H})\rightarrow A^{2}_{\alpha}(\mathbb{H})$ on
 $A^{2}_{\alpha}(\mathbb{H})$ is defined as
 $T_{f}=PM_{f}$ where
 $P:L^{2}(\mathbb{H},dA_{\alpha})\rightarrow
 A^{2}_{\alpha}(\mathbb{H})$ is the orthogonal projection,
 $M_{f}:L^{2}(\mathbb{H},dA_{\alpha})\rightarrow
 L^{2}(\mathbb{H},dA_{\alpha})$ is the multiplication operator and
 $f\in L^{\infty}(\mathbb{H})$.

 In \cite{Zh} Zhu introduced the space of functions
 $VMO_{\partial}(\mathbb{D})$ of ``vanishing mean oscillation near
 the boundary" which is defined as the set of functions $f\in
 L^{1}(\mathbb{D},dA)$ satisfying
 $$\lim_{\mid z\mid\rightarrow 1^{-}}\frac{1}{\mid Q_{z}\mid}\int_{Q_{z}}\mid
 f(w)-\frac{1}{\mid Q_{z}\mid}\int_{Q_{z}}f(u)dA(u)\mid dA(w)=0$$
 where $Q_{z}=\{w\in\mathbb{D}:\mid w\mid\geq\mid z\mid,\mid\arg
 w-\arg z\mid\leq 1-\mid z\mid\}$ and $\mid Q_{z}\mid=(1+\mid
 z\mid)(1-\mid z\mid)^{2}$ is the $dA$ measure of $Q_{z}$. Zhu
 showed for the case $\alpha=0$ that for $f\in
 L^{\infty}(\mathbb{D})$, the semi-commutator
 $T_{gf}-T_{g}T_ {f}\in K(A^{2})$ is compact on
 $A^{2}$ $\forall g\in L^{\infty}(\mathbb{D})$ if and only if
 $f\in VMO_{\partial}(\mathbb{D})$. Although Zhu proved this
 fact for $\alpha=0$, his proofs work for the general weighted
 case i.e. for $\alpha\neq 0$. Zhu also introduced the space of
 functions $ESV(\mathbb{D})$ of ``eventually slowly varying" which
 is defined as the set of functions $f\in L^{\infty}(\mathbb{D})$
 satisfying for any $\varepsilon>0$ and $\kappa\in (0,1)$ there is
 $\delta_{0}>0$ such that $$\mid f(z)-f(w)\mid<\varepsilon$$
 whenever $\mid z\mid,\mid w\mid\in [1-\delta,1-\kappa\delta]$,
 $\delta<\delta_{0}$ and $\mid\arg z-\arg w\mid\leq\max(1-\mid
 z\mid,1-\mid w\mid)$. Zhu proved that
 $$VMO_{\partial}(\mathbb{D})\cap
 H^{\infty}(\mathbb{D})=ESV(\mathbb{D})\cap
 H^{\infty}(\mathbb{D}).$$

 We finish the Preliminaries section by exhibiting an isometric
 isomorphism between $A^{2}_{\alpha}(\mathbb{D})\longleftrightarrow$
 $A^{2}_{\alpha}(\mathbb{H})$ where $$A_{\alpha}^{2}(\mathbb{D}) = \{
 f:\mathbb{D}\rightarrow\mathbb{C}:\textrm{f is analytic and}\qquad\int_{\mathbb{D}}\mid
 f(z)\mid^{2}(1-\mid z\mid^{2})^{\alpha}dA(z) <\infty \}$$ and
 $\mathbb{D}$ is the unit disc. The map
 $\Phi :A^{2}_{\alpha}(\mathbb{D})\rightarrow$
 $A^{2}_{\alpha}(\mathbb{H})$, $$\Phi(f)(z) = \frac{2^{\alpha +1}}{(z+i)^{\alpha
 +2}}f\bigg(\frac{z-i}{z+i}\bigg)$$ is an isometric isomorphism. It
 is of interest to us what the composition operators become under
 intertwining with this isomorphism, i.e. for
 $\varphi:\mathbb{D}\rightarrow$ $\mathbb{D}$ what is
 $\Phi C_{\varphi}\Phi^{-1}$?: We have the following answer to
 this question: $$\Phi C_{\varphi}\Phi^{-1} =
 M_{\tau^{2}}C_{\tilde{\varphi}}$$ where $M_{\tau^{2}}f(z)=$ $\tau(z)^{2}f(z)$
 is the multiplication operator, $\tau(z)=$
 $\frac{\tilde{\varphi}(z)+i}{z+i}$, $\tilde{\varphi}=$
 $C^{-1}\circ\varphi\circ C$ and $C(z)=$ $\frac{z-i}{z+i}$ is the
 Cayley transform. This gives us the boundedness of
 $C_{\varphi}:A^{2}_{\alpha}(\mathbb{H})$ $\rightarrow$
 $A^{2}_{\alpha}(\mathbb{H})$ for $\varphi(z)=$ $pz+\psi(z)$ where $p>$
 $0$, $\psi\in$ $H^{\infty}$ and the closure of the image
 $\overline{\psi(\mathbb{H})}\subset\subset$ $\mathbb{H}$ is compact in
 $\mathbb{H}$:

 Let $\theta:\mathbb{D}\rightarrow$ $\mathbb{D}$ be an analytic
 self-map of $\mathbb{D}$ such that $\tilde{\theta}=$ $C^{-1}\circ\theta\circ
 C=$ $\varphi$ then we have $\Phi C_{\theta}\Phi^{-1}=$
 $M_{\tau^{2}}C_{\varphi}$ where $\tau(z)=$
 $\frac{\varphi(z)+i}{z+i}$. If $\varphi(z)=$ $pz+\psi(z)$ with $p>$
 $0$, $\psi\in$ $H^{\infty}$ and $\Im(\psi(z))>$ $\delta>$ $0$ then
 $M_{\frac{1}{\tau^{2}}}$ is a bounded
 operator. Since $C_{\theta}$ is always bounded(see \cite{Cow-Mac}), this implies that $\Phi C_{\theta}\Phi^{-1}$
 is also bounded and we conclude
 that $C_{\varphi}$ is bounded on $A^{2}_{\alpha}(\mathbb{H})$(see also
 \cite{Ell-W}). We also observe that for any $f\in
 L^{\infty}(\mathbb{D})$ we have
 $$\Phi\circ T_{f}\circ\Phi^{-1}=T_{f\circ\mathfrak{C}} $$

 \section{Approximation Scheme for Composition Operators
 on Weighted Bergman Spaces of the Upper Half-Plane}

 In this section we develop an approximation scheme for composition
 operators induced by the maps of the form $\phi(z)=$ $pz+\psi(z)$
 where $p>$ $0$ and $\psi\in$ $H^{\infty}$ such that the closure of
 the image $\overline{\psi(\mathbb{H})}\subset\subset$ $\mathbb{H}$
 is compact in $\mathbb{H}$, by linear combinations of $T_{\psi}$
 and Fourier multipliers where $T_{\psi}$ is the Toeplitz operator
 with symbol $\psi$. By the preceding section we know that these
 maps induce bounded composition operators on $A^{2}_{\alpha}(\mathbb{H})$. In establishing
 this approximation scheme our main tools are the integral representation formulas which come from the
 fact that these spaces are reproducing kernel Hilbert spaces. We
 begin with a simple geometric lemma that will be helpful in our
 task:

 \begin{lemma} Let $K\subset$ $\mathbb{H}$ be a compact
 subset of $\mathbb{H}$. Then $\exists$ $\beta\in$ $\mathbb{R}^{+}$
 such that $\sup\{\mid\frac{\beta i-z}{\beta}\mid:z\in K\}<$
 $\delta<$ $1$ for some $\delta\in$ $(0,1)$
 \end{lemma}

 \begin{proof}
 See \cite{Gul}
\end{proof}
  We also need the following lemma which characterizes certain
  integral operators as Fourier multipliers
\begin{lemma}
Let $\alpha\in\mathbb{R}$ s.t. $\alpha>-1$ and
$\beta\in\mathbb{H}$. Then the operator
$M_{n}:A^{2}_{\alpha}(\mathbb{H})\rightarrow
A^{2}_{\alpha}(\mathbb{H})$ defined as
$$(M_{n}f)(z)=\frac{-1}{\pi}\int_{\mathbb{H}}\frac{f(w)dA_{\alpha}(w)}{(\bar{w}-z-\beta)^{n+\alpha+2}}$$
where $n\in\mathbb{N}\setminus\{0\}$, is the Fourier multiplier
$D_{\phi_{n}}$ where
$$\phi_{n}(t)=\frac{(2\pi it)^{n}e^{2\pi i\beta
t}}{(\alpha+2)(\alpha+3)...(\alpha+n+1)}$$ i.e.
$M_{n}=D_{\phi_{n}}$.
\end{lemma}

\begin{proof}
Let $f\in L^{2}_{\alpha+1}(\mathbb{R}^{+})$ then for
\begin{equation}
g(z)=\mathcal{F}^{-1}(f)(z)=\int_{0}^{\infty}f(t)e^{2\pi itz}dt
\end{equation}
 we have $g\in A^{2}_{\alpha}(\mathbb{H})$. Since
$A^{2}_{\alpha}(\mathbb{H})$ is a reproducing kernel Hilbert space
we have
$$g(z_{0})=\frac{-1}{\pi}\int_{\mathbb{H}}\frac{f(w)dA_{\alpha}(w)}{(\bar{w}-z_{0})^{\alpha+2}}$$
Hence differentiating under the integral sign $n$ times and
substituting $z_{0}=z+\beta$ we have
$$M_{n}g(z)=\frac{1}{(\alpha+2)(\alpha+3)...(\alpha+n+1)}g^{(n)}(z+\beta)$$
Combining this with equation (4) we have
$$M_{n}g(z)=\frac{1}{(\alpha+2)(\alpha+3)...(\alpha+n+1)}\frac{d^{n}}{dz^{n}}(\int_{0}^{\infty}f(t)e^{2\pi
it(z+\beta)}dt)$$ Interchanging the differentiation and
integration in the above equation we have
$$M_{n}g(z)=\frac{1}{(\alpha+2)(\alpha+3)...(\alpha+n+1)}\int_{0}^{\infty}(2\pi
it)^{n}e^{2\pi i\beta t}f(t)e^{2\pi
itz}dt=\mathcal{F}^{-1}(\phi_{n}f)(z)$$ where
$$\phi_{n}(t)=\frac{(2\pi it)^{n}e^{2\pi i\beta
t}}{(\alpha+2)(\alpha+3)...(\alpha+n+1)}$$ This means that
$$M_{n}(\mathcal{F}^{-1}(f))(z)=\mathcal{F}^{-1}(\phi_{n}f)(z)$$
which implies that
$$M_{n}=\mathcal{F}^{-1}M_{\phi_{n}}\mathcal{F}=D_{\phi_{n}}$$
\end{proof}

 \begin{proposition}
    Let $\varphi:\mathbb{H}\rightarrow$
    $\mathbb{H}$ be an analytic self-map of the form $\varphi(z)=$
    $pz+\psi(z)$ where $p>$ $0$ and $\psi\in$ $H^{\infty}$ with the
    closure of the image $\overline{\psi(\mathbb{H})}\subset\subset$
    $\mathbb{H}$ is compact in $\mathbb{H}$. Then $\exists$ $\beta>$
    $0$ such that for $C_{\varphi}:A_{\alpha}^{2}\rightarrow$
    $A_{\alpha}^{2}$ we have $$C_{\varphi} =
    \sum_{n=0}^{\infty}\frac{\Gamma(n+2+\alpha)}{n!\Gamma(\alpha+2)}T_{\tau^{n}}D_{\phi_{n}}V_{p}$$
   where the convergence of the series is in operator norm,
   $T_{\tau^{n}}f(z)=$ $\tau^{n}(z)f(z)$, $\tau(z)=$
   $\tilde{\psi}(z)-i\beta$, $\tilde{\psi}(z)=$
   $\psi(\frac{z}{p})$, $V_{p}f(z)=$ $f(pz)$ is the dilation by $p$
   and $\phi_{n}(t)=$ $\frac{(2\pi it)^{n}e^{-2\pi\beta
t}}{(\alpha+2)(\alpha+3)...(\alpha+n+1)}$ for $n\geq 1$ and
$\phi_{0}(t)=e^{-2\pi\beta t}$.
   \end{proposition}

\begin{proof}
 The integral representation formula is as follows: For $f\in$
 $A^{2}_{\alpha}(\mathbb{H})$ we have $$f(z) =
 \frac{-1}{\pi}\int_{\mathbb{H}}\frac{f(w)dA_{\alpha}(w)}{(\bar{w}-z)^{\alpha+2}}$$
 where $dA_{\alpha}(w)=$ $(\Im(w))^{\alpha}dA(w)$ is a translation
 invariant measure on $\mathbb{H}$.

 For $\varphi:\mathbb{H}\rightarrow$ $\mathbb{H}$ an analytic
 self-map of $\mathbb{H}$, we can insert the substitution
 $z\rightarrow$ $\varphi(z)$ in order to get an integral
 representation of the composition operator
 $C_{\varphi}:A_{\alpha}^{2}\rightarrow$ $A_{\alpha}^{2}$:
 $$C_{\varphi}(f)(z) =
 \frac{-1}{\pi}\int_{\mathbb{H}}\frac{f(w)dA_{\alpha}(w)}{(\bar{w}-\varphi(z))^{\alpha+2}}$$

  Let $\varphi(z)=$ $pz+\psi(z)$ where $\psi\in$
    $H^{\infty}$ with $\Im(\psi(z))>$ $\delta>$ $0$ $\forall$ $z\in$
    $\mathbb{H}$ and $p>$ $0$. Then for $C_{\varphi}:A^{2}_{\alpha}\rightarrow$ $A^{2}_{\alpha}$ we have
\begin{equation}
(C_{\varphi}V_{\frac{1}{p}})f(z) =
\frac{-1}{\pi}\int_{\mathbb{H}}\frac{f(w)dA_{\alpha}(w)}{(\bar{w}-z-\tilde{\psi}(z))^{\alpha+2}}
\end{equation}
where $\tilde{\psi}(z)=$ $\psi(\frac{z}{p})$. We look at
$$\frac{1}{(\bar{w}-z-\tilde{\psi}(z))^{\alpha+2}} =
    \frac{1}{(\bar{w}-z-i\beta-(\tilde{\psi}(z)-i\beta))^{\alpha+2}}$$
    $$=\frac{1}{(\bar{w}-z-i\beta)^{\alpha+2}\bigg(1-\frac{\tilde{\psi}(z)-i\beta}{\bar{w}-z-i\beta}\bigg)^{\alpha+2}}$$
    We apply Lemma 1 to have $\beta>$ $0$ such that
    $$\mid\frac{\tilde{\psi(z)}-i\beta}{\bar{w}-z-i\beta}\mid<\delta<1$$
    Here we have the geometric
    series formula as
    $$\frac{1}{\bigg(1-\frac{\tilde{\psi}(z)-i\beta}{\bar{w}-z-i\beta}\bigg)^{\alpha+2}}
    =
    \sum_{n=0}^{\infty}\frac{\Gamma(n+2+\alpha)}{n!\Gamma(\alpha+2)}\bigg(\frac{\tilde{\psi}(z)-i\beta}{\bar{w}-z-i\beta}\bigg)^{n}$$
    $$=\sum_{n=0}^{M}\frac{\Gamma(n+2+\alpha)}{n!\Gamma(\alpha+2)}\bigg(\frac{\tilde{\psi}(z)-i\beta}{\bar{w}-z-i\beta}\bigg)^{n}
    + Q_{M+1}(w,z)$$ Inserting this into (5) we have
    $$C_{\varphi}f(z) =
    \sum_{n=0}^{M}\frac{\Gamma(n+2+\alpha)}{n!\Gamma(\alpha+2)}T_{\tau^{n}}D_{\phi_{n}}f(z)
    +
    \int_{\mathbb{H}}\frac{Q_{M+1}(w,z)f(w)dA_{\alpha}(w)}{(\bar{w}-z-i\beta)^{\alpha+2}}$$
    where $T_{\tau^{n}}f(z)=$ $\tau^{n}(z)f(z)$, $\tau(z)=$
    $\tilde{\psi}(z)-i\beta$ and $D_{\phi_{n}}$ is the Fourier multiplier
    $D_{\phi_{n}}f(z)=$
    $\frac{-1}{\pi}\int_{\mathbb{H}}\frac{f(w)dA_{\alpha}(w)}{(\bar{w}-z-i\beta)^{n+\alpha+2}}$

   For the operator $$R_{M+1}f(z) =
    \int_{\mathbb{H}}\frac{Q_{M+1}(w,z)f(w)dA_{\alpha}(w)}{(\bar{w}-z-i\beta)^{\alpha+2}}$$
    we have $\parallel R_{M+1}\parallel\leq$ $\parallel
    Q_{M+1}\parallel_{\infty}\parallel S_{i\beta}\parallel$ where
    $S_{i\beta}f(z)=$ $f(z+i\beta)$ and $\parallel
    Q_{M+1}\parallel_{\infty}=$ $\sup_{(z,w)\in\mathbb{H}^{2}}\mid
    Q_{M+1}(w,z)\mid$. Since $\parallel Q_{M+1}\parallel_{\infty}\rightarrow$
    $0$ as $M\rightarrow$ $\infty$ we have $\parallel
    R_{M+1}\parallel\rightarrow$ $0$ as $M\rightarrow$ $\infty$.
\end{proof}

\section{A C*-algebra of Operators on $A^{2}_{\alpha}(\mathbb{H})$ }
  In the preceding section we have shown that ``quasi-parabolic''
composition operators on the upper half-plane lie in the
C*-algebra generated by certain Toeplitz operators and Fourier
multipliers. In this section we will identify the character space
of the C*-algebra generated by Toeplitz operators with a class of
symbols and Fourier multipliers.

  In the Hardy space case we observed that if $\varphi\in QC$ and
$\theta\in C([0,\infty])$ then the commutator
$T_{\varphi}D_{\theta}-D_{\theta}T_{\varphi}\in K(H^{2})$ is
compact on the Hardy space $H^{2}$. In the weighted Bergman space
case which we are considering in this paper, $QC$ will be replaced
by $VMO_{\partial}$. The upper half-plane versions of
$VMO_{\partial}$ and $ESV$ are defined as follows:
$$VMO_{\partial}(\mathbb{H})=\{f\circ\mathfrak{C}:f\in VMO_{\partial}(\mathbb{D})\}
$$ and $$ESV(\mathbb{H})=\{f\circ\mathfrak{C}:f\in
ESV(\mathbb{D})\}.$$ It is not difficult to see that by definition
and Zhu's result(\cite{Zh}) we have
$$VMO_{\partial}(\mathbb{H})\cap H^{\infty}(\mathbb{H})= ESV(\mathbb{H})\cap H^{\infty}(\mathbb{H})$$

\begin{lemma}
Let $\psi\in ESV(\mathbb{H})\cap
H^{\infty}(\mathbb{H})=VMO_{\partial}(\mathbb{H})\cap
H^{\infty}(\mathbb{H})$
then\\
$\mathcal{C}_{\psi,a}:A^{2}_{\alpha}(\mathbb{H})\rightarrow
A^{2}_{\alpha}(\mathbb{H})$ is compact where
$$\mathcal{C}_{\psi,a}(f)(z)=(\psi(z+a)-\psi(z))f(z+a)$$ where
$a\in\mathbb{H}$.
\end{lemma}

\begin{proof}
Since $\psi\in ESV$, $\forall\varepsilon>0$ there exists
$m\in\mathbb{N}$ and a compact subset $K\subset\mathbb{H}$ so that
$\mid\psi(z+\frac{a}{m})-\psi(z)\mid<\frac{\varepsilon}{m}$
$\forall z\not\in K$. So for all $k\in\{1,2,..,m\}$ there is a
compact subset $K_{k}\subset\mathbb{H}$ so that $\forall z\not\in
K_{k}$ we have
$$\mid \psi(z+\frac{ka}{m})-\psi(z+\frac{(k-1)a}{m})\mid<\frac{\varepsilon}{m}$$
Hence there is a compact set $K=\cup_{k=1}^{m}K_{k}$ so that
$\forall z\not\in K$ we have
$$\mid \psi(z+a)-\psi(z)\mid\leq\sum_{k=0}^{m-1}\mid \psi(z+\frac{(k+1)a}{m})-\psi(z+\frac{ka}{m})\mid<m\sum_{k=0}^{m-1}\frac{\varepsilon}{m}=\varepsilon$$
So we have $\forall\varepsilon>0$, there is a compact subset
$K\subset\mathbb{H}$ so that
$$\mid \psi(z+a)-\psi(z)\mid<\varepsilon\quad\forall z\not\in K$$

Let $\{f_{n}\}_{n=1}^{\infty}\subset A^{2}_{\alpha}(\mathbb{H})$
so that $\parallel f_{n}\parallel_{A^{2}_{\alpha}}\leq 1$ $\forall
n\in\mathbb{N}$ and let $g_{n}=\mathcal{C}_{\psi,a}(f_{n})$.

Let $K_{m}\subset\mathbb{H}$ be a sequence of compact subsets of
$\mathbb{H}$ so that $K_{m}\subset (K_{m+1})^{\circ}$ and
$\bigcup_{m\in\mathbb{N}}K_{m}=\mathbb{H}$. Since $\{g_{n}\}$ is
equi-bounded on $K_{1}$ by Montel's theorem it has a subsequence
$\{g_{n_{k}}\}$ so that it converges uniformly on $K_{1}$. Since
$\{g_{n_{k}}\}$ is equi-bounded on $K_{2}$ it has a further
subsequence $\{g_{n_{k_{l}}}\}$ so that it converges uniformly on
$K_{2}$. Proceeding in this way using Cantor's diagonal argument
one can extract a subsequence $\{g_{k}\}$ so that $\{g_{k}\}$ is
uniformly convergent on $K_{m}$ $\forall m\in\mathbb{N}$.

Let $\varepsilon>0$ be given then $\exists$ $m\in\mathbb{N}$ so
that $$\mid\psi(z+a)-\psi(z)\mid^{2}<\varepsilon\quad\forall
z\not\in K_{m}$$ Consider
\begin{eqnarray*}
& &\parallel g_{k}-g_{l}\parallel^{2}_{A^{2}_{\alpha}}=\int_{\mathbb{H}}\mid\psi(z+a)-\psi(z)\mid^{2}\mid(f_{k}-f_{l})(z+a)\mid^{2}dA_{\alpha}(z)\\
& &=\int_{K_{m}}\mid\psi(z+a)-\psi(z)\mid^{2}\mid(f_{k}-f_{l})(z+a)\mid^{2}dA_{\alpha}(z)\\
& &+\int_{\mathbb{H}\setminus K_{m}}\mid\psi(z+a)-\psi(z)\mid^{2}\mid(f_{k}-f_{l})(z+a)\mid^{2}dA_{\alpha}(z)\\
& &=\int_{K_{m}}\mid g_{k}(z)-g_{l}(z)\mid^{2}dA_{\alpha}(z)\\
& &+\int_{\mathbb{H}\setminus K_{m}}\mid\psi(z+a)-\psi(z)\mid^{2}\mid(f_{k}-f_{l})(z+a)\mid^{2}dA_{\alpha}(z)\\
\end{eqnarray*}
Since $\{g_{k}\}$ is uniformly convergent on $K_{m}$, $\exists$
$n_{0}\in\mathbb{N}$ so that $\forall k,l>n_{0}$ $$\parallel
g_{k}-g_{l}\parallel_{K_{m}}<\frac{\varepsilon}{\lambda_{\alpha}(K_{m})}$$
where $\lambda_{\alpha}(K_{m})=\int_{K_{m}}dA_{\alpha}(z)$ is the
$dA_{\alpha}$ measure of $K_{m}$. Hence $$\int_{K_{m}}\mid
g_{k}(z)-g_{l}(z)\mid^{2}dA_{\alpha}(z)<\varepsilon\quad\forall
k,l>n_{0}$$ and we have $$\int_{\mathbb{H}\setminus
K_{m}}\mid\psi(z+a)-\psi(z)\mid^{2}\mid(f_{k}-f_{l})(z+a)\mid^{2}dA_{\alpha}(z)\leq
2\parallel f_{k}\parallel_{A^{2}_{\alpha}}\leq 2\varepsilon$$
since $\parallel f_{k}\parallel_{A^{2}_{\alpha}}\leq 1$ $\forall
k\in\mathbb{N}$. Hence $\forall\varepsilon>0$ $\exists
n_{0}\in\mathbb{N}$ so that $\forall k,l>n_{0}$ $$\parallel
g_{k}-g_{l}\parallel_{A^{2}_{\alpha}}<3\varepsilon$$ This implies
that $\{g_{k}\}$ is a Cauchy sequence in $A^{2}_{\alpha}$. So for
any sequence $\{f_{n}\}\subset A^{2}_{\alpha}$ satisfying
$\parallel f_{n}\parallel_{A^{2}_{\alpha}}\leq 1$,
$g_{n}=\mathcal{C}_{\psi,a}(f_{n})$ has a convergent subsequence
in $A^{2}_{\alpha}$. This implies that $\mathcal{C}_{\psi,a}$ is
compact on $A^{2}_{\alpha}$.
\end{proof}
\begin{corollary}
Let $\psi\in VMO_{\partial}\cap L^{\infty}(\mathbb{H})$ and
$\theta\in C([0,\infty])$ then the commutator
$T_{\psi}D_{\theta}-D_{\theta}T_{\psi}\in
K(A^{2}_{\alpha}(\mathbb{H}))$ is compact on $A^{2}_{\alpha}$.
\end{corollary}
\begin{proof}
Let $S_{a}:A^{2}_{\alpha}\rightarrow A^{2}_{\alpha}$ be the
translation operator $S_{a}f(z)=f(z+a)$ where $a\in\mathbb{H}$.
Then we observe that $S_{a}=D_{\phi_{a}}$ where
$\phi_{a}(t)=e^{2\pi iat}$. Since for $\psi\in VMO_{\partial}\cap
H^{\infty}(\mathbb{H})$
$\mathcal{C}_{\psi,a}=S_{a}T_{\psi}-T_{\psi}S_{a}$ we have
$T_{\psi}D_{\phi_{a}}-D_{\phi_{a}}T_{\psi}\in
K(A^{2}_{\alpha}(\mathbb{H})$ by lemma 4. We also observe that
$S_{-\bar{a}}=D_{\phi_{-\bar{a}}}$ and hence
$S_{-\bar{a}}T_{\psi}-T_{\psi}S_{-\bar{a}}\in K(A^{2}_{\alpha}$
since $-\bar{a}\in\mathbb{H}$ for $a\in\mathbb{H}$. By
Stone-Weierstrass theorem the set of functions of the form
$p(\phi_{a},\phi_{-\bar{a}})$ where
$p(z,w)=\sum_{k,l=0}^{n,m}c_{k,l}z^{k}w^{l}$ is dense in
$C([0,\infty])$, hence $T_{\psi}D_{\theta}-D_{\theta}T_{\psi}\in
K(A^{2}_{\alpha})$ $\forall\theta\in C([0,\infty])$. Since
$VMO_{\partial}\cap L^{\infty}(\mathbb{H})$ is generated by
functions $\psi$ and $\bar{\psi}$ where $\psi\in
VMO_{\partial}\cap H^{\infty}(\mathbb{H})$ we have
$T_{\psi}D_{\theta}-D_{\theta}T_{\psi}\in K(A^{2}_{\alpha})$
$\forall\psi\in VMO_{\partial}\cap L^{\infty}$ and
$\forall\theta\in C([0,\infty])$.
\end{proof}

Now we are ready to construct our C*-algebra of operators. Let
$$\mathcal{T}(VMO_{\partial})=C^{\ast}(\{T_{f}:f\in
VMO_{\partial}(\mathbb{H})\cap
L^{\infty}(\mathbb{H})\})\subset\mathcal{B}(A^{2}_{\alpha}(\mathbb{H}))$$
be the C*-algebra of Toeplitz operators on
$A^{2}_{\alpha}(\mathbb{H})$ with symbols in $VMO_{\partial}$. By
Zhu's result(\cite{Zh}), since the commutators
$T_{f}T_{g}-T_{g}T_{f}\in K(A^{2}_{\alpha}(\mathbb{H}))$ are
compact for all $f,g\in VMO_{\partial}\cap
L^{\infty}(\mathbb{H})$, it is easy to see that the quotient
C*-algebra
$\mathcal{T}(VMO_{\partial})/K(A^{2}_{\alpha}(\mathbb{H}))$ is a
unital commutative C*-algebra. Let $\mathcal{M}$ be the maximal
ideal space of
$\mathcal{T}(VMO_{\partial})/K(A^{2}_{\alpha}(\mathbb{H}))$. It is
a well known fact in the theory of Toeplitz operators on Bergman
spaces that
$\mathcal{T}(C(\overline{\mathbb{D}}))/K(A^{2}_{\alpha}(\mathbb{D}))$
is isometrically isomorphic to $C(\mathbb{T})$ and the isometric
isomorphism is given by the correspondence $[T_{f}]\rightarrow
f|_{\mathbb{T}}$(see \cite{Vas}). This fact can be carried over to
the upper half-plane by using the Cayley transform: Let
$\overline{\mathbb{H}}=\{z\in\mathbb{C}:\Im(z)\geq 0\}$ and
$\dot{\overline{\mathbb{H}}}$ be the one point compactification of
$\overline{\mathbb{H}}$, then the correspondence
$[T_{f}]\rightarrow f|_{\dot{\mathbb{R}}}$ is an isometric
isomorphism between
$\mathcal{T}(C(\dot{\overline{\mathbb{H}}}))/K(A^{2}_{\alpha}(\mathbb{H}))$
and $C(\dot{\mathbb{R}})$. Since
$\mathcal{T}(C(\dot{\overline{\mathbb{H}}}))/K(A^{2}_{\alpha}(\mathbb{H}))$
is a subalgebra of
$\mathcal{T}(VMO_{\partial})/K(A^{2}_{\alpha}(\mathbb{H}))$ and
$\mathcal{T}(C(\dot{\overline{\mathbb{H}}}))/K(A^{2}_{\alpha}(\mathbb{H}))$
is isometrically isomorphic to $C(\dot{\mathbb{R}})$, the maximal
ideal space $\mathcal{M}$ can be thought as fibered over
$\dot{\mathbb{R}}$. For $x\in\dot{\mathbb{R}}$, let
$$\mathcal{M}_{x}=\{\phi\in\mathcal{M}:\phi|_{\mathcal{T}(C(\dot{\overline{\mathbb{H}}}))/K(A^{2}_{\alpha}(\mathbb{H}))}=\delta_{x}\}$$
where $\delta_{x}([T_{f}])=f(x)$. Then we have
$$\mathcal{M}=\bigcup_{x\in\dot{\mathbb{R}}}\mathcal{M}_{x}$$ and
for $x_{1}\neq x_{2}$ we have
$\mathcal{M}_{x_{1}}\cap\mathcal{M}_{x_{2}}=\emptyset$. Now let
$F_{C([0,\infty])}^{A^{2}_{\alpha}}$ be the C*-algebra generated
by Fourier multipliers with symbols in $C([0,\infty])$, our
C*-algebra is the following
$$\Psi(VMO_{\partial},C([0,\infty]))=C^{\ast}(\mathcal{T}(VMO_{\partial})\cup F_{C([0,\infty])}^{A^{2}_{\alpha}})\subset\mathcal{B}(A^{2}_{\alpha}(\mathbb{H}))$$
By corollary 5,
$\Psi(VMO_{\partial},C([0,\infty]))/K(A^{2}_{\alpha}(\mathbb{H}))$
is a unital commutative C*-algebra. It is of interest to ask for
its maximal ideal space.

We will use the following theorem of Power (see \cite{Pow1} and
\cite{Pow2}) to characterize its maximal ideal space:

\begin{Power}
    Let $C_{1}$, $C_{2}$ be two C*-subalgebras of $B(H)$ with identity,
    where $H$ is a separable Hilbert space, such that $M(C_{i})\neq$
    $\emptyset$, where $M(C_{i})$ is the space of multiplicative linear
    functionals of $C_{i}$, $i= 1,\,2$ and let $C$ be the C*-algebra that
    they generate. Then for the commutative C*-algebra $\tilde{C}=$
    $C/com(C)$ we have $M(\tilde{C})=$ $P(C_{1},C_{2})\subset$
    $M(C_{1})\times M(C_{2})$, where $P(C_{1},C_{2})$ is defined to be
    the set of points $(x_{1},x_{2})\in$ $M(C_{1})\times M(C_{2})$
    satisfying the condition: \\
    \quad Given $0\leq a_{1} \leq 1$, $0 \leq a_{2} \leq 1$, $a_{1}\in C_{1}$, $a_{2}\in C_{2}$
\begin{equation*}
    x_{i}(a_{i})=1\quad\textrm{with}\quad
        i=1,2\quad\Rightarrow\quad\| a_{1}a_{2}\|=1.
\end{equation*}
    \label{thmpower}
\end{Power}

\begin{proof}
See \cite{Pow2}.
\end{proof}

\begin{theorem}
Let
$$\Psi(VMO_{\partial},C([0,\infty]))=C^{\ast}(\mathcal{T}(VMO_{\partial})\cup
F_{C([0,\infty])}^{A^{2}_{\alpha}})\subset\mathcal{B}(A^{2}_{\alpha}(\mathbb{H}))$$
then
$\Psi(VMO_{\partial},C([0,\infty]))/K(A^{2}_{\alpha}(\mathbb{H}))$
is a unital commutative C*-algebra. For its maximal ideal space
$M(\Psi)$ we have
$$M(\Psi)\cong(\mathcal{M}\times\{\infty\})\cup(\mathcal{M}_{\infty}\times
[0,\infty])$$ where
$$\mathcal{M}_{\infty}=\{\phi\in\mathcal{M}:\phi|_{\mathcal{T}(C(\dot{\overline{\mathbb{H}}}))/K(A^{2}_{\alpha}(\mathbb{H}))}=\delta_{\infty}\}$$
is the fiber of $\mathcal{M}$ at $\infty$ with
$\delta_{\infty}([T_{f}])=f(\infty)$.
\end{theorem}
\begin{proof}
We will use Power's theorem. In our case,
$$H=A^{2}_{\alpha}(\mathbb{H}), C_{1}=\mathcal{T}(VMO_{\partial}),C_{2}=F_{C([0,\infty])}^{A^{2}_{\alpha}},\textrm{and},\tilde{C}=\Psi(VMO_{\partial},C([0,\infty]))/K(A^{2}_{\alpha}(\mathbb{H})).$$
    We have
\begin{equation*}
    M(C_{1})=\mathcal{M}\quad\textrm{and}\quad M(C_{2})=[0,\infty].
\end{equation*}
    So we need to determine $(x,y)\in$ $\mathcal{M}\times [0,\infty]$ so
    that for all $f\in$ $VMO_{\partial}(\mathbb{H})$ and $\vartheta\in$
    $C([0,\infty])$ with $0 < f, \vartheta\leq 1$, we have
\begin{equation*}
    \hat{T_{f}}(x)=\vartheta(y)=1\Rightarrow\parallel
    T_{f}D_{\vartheta}\parallel=1\quad\textrm{or}\quad\parallel
    D_{\vartheta}T_{f}\parallel=1.
\end{equation*}

     Let $x\in\mathcal{M}$ such that $x\in \mathcal{M}_{t}$ with $t\neq \infty$
    and $y\in$ $[0,\infty)$. Choose $f\in$ $C(\dot{\overline{\mathbb{H}}})$ and
    $\vartheta\in$ $C([0,\infty])$ such that
\begin{equation*}
    \hat{T_{f}}(x)= f(t)=\vartheta(y)=1,\quad 0\leq f\leq 1,\quad
    0\leq\vartheta\leq 1,\quad f(z)<1
\end{equation*}
    for all $z\in$
    $\overline{\mathbb{H}}\backslash\{t\}$ and\quad $\vartheta(w) < 1$\quad for
    all $w\in$ $[0,\infty]\backslash\{y\}$\quad, where both $f$ and
    $\vartheta$ have compact supports. Consider
    $D_{\vartheta}M_{f}:L^{2}_{\alpha}(\mathbb{H})\rightarrow
    L^{2}_{\alpha}(\mathbb{H})$: we have
\begin{eqnarray*}
& &(D_{\vartheta}M_{f})(g)(z)=\frac{2^{\alpha +1}}{\Gamma(\alpha +1)}\int_{0}^{\infty}\vartheta(t)e^{2\pi it}(\int_{\mathbb{H}}t^{1+\alpha}e^{-2\pi i\bar{\zeta}}f(\zeta)g(\zeta)dA_{\alpha}(\zeta))dt\\
& &=\frac{2^{\alpha +1}}{\Gamma(\alpha +1)}\int_{\mathbb{H}}(f(\zeta)\int_{0}^{\infty}\vartheta(t)t^{\alpha +1}e^{2\pi it(z-\bar{\zeta})}dt)g(\zeta)dA_{\alpha}(\zeta)\\
& &=\int_{\mathbb{H}}k(z,\zeta)g(\zeta)dA_{\alpha}(z)
\end{eqnarray*}
where $k(z,\zeta)=\frac{2^{\alpha +1}}{\Gamma(\alpha
+1)}f(\zeta)\int_{0}^{\infty}\vartheta(t)t^{\alpha +1}e^{2\pi
it(z-\bar{\zeta})}dt$. Since $f$ and $\vartheta$ have compact
supports we have
$$\int_{\mathbb{H}}\int_{\mathbb{H}}\mid k(z,\zeta)\mid^{2}dA_{\alpha}(z)dA_{\alpha}(\zeta)\leq\frac{2^{\alpha +1}}{\Gamma(\alpha +1)}\parallel f\parallel_{\infty}^{2}A_{\alpha}(K_{1})\parallel\vartheta\parallel_{L^{2}(\mathbb{R})}^{2}A_{\alpha}(K_{2})$$
where $K_{1},K_{2}\subset\overline{\mathbb{H}}$ are supports of
$f$ and $\vartheta$ respectively, and
$$A_{\alpha}(K)=\int_{K}dA_{\alpha}(z)$$ is the $dA_{\alpha}$
measure of the compact subset $K\subset\overline{\mathbb{H}}$.
This implies that $D_{\vartheta}M_{f}$ is Hilbert-Schmidt on
$L^{2}_{\alpha}(\mathbb{H})$ and hence is compact. Since
$$\parallel
T_{f}D_{\vartheta}\parallel_{A^{2}_{\alpha}}\leq\parallel
M_{f}D_{\vartheta}\parallel_{L^{2}_{\alpha}}$$ and
$$\parallel M_{f}D_{\vartheta}\parallel_{L^{2}_{\alpha}}=\parallel(M_{f}D_{\vartheta})^{\ast}\parallel_{L^{2}_{\alpha}}=\parallel D_{\vartheta}M_{f}\parallel_{L^{2}_{\alpha}}$$
we have by C*-equality
\begin{eqnarray*}
& &\parallel T_{f}D_{\vartheta}\parallel_{A^{2}_{\alpha}}^{2}\leq\parallel M_{f}D_{\vartheta}\parallel_{L^{2}_{\alpha}}^{2}=\parallel D_{\vartheta}M_{f}\parallel_{L^{2}_{\alpha}}^{2}\\
& &=\parallel(D_{\vartheta}M_{f})^{\ast}(D_{\vartheta}M_{f})\parallel_{L^{2}_{\alpha}}=\parallel M_{f}D_{\vartheta}^{2}M_{f}\parallel_{L^{2}_{\alpha}}\\
\end{eqnarray*}
Since $M_{f}D_{\vartheta}^{2}M_{f}$ is a compact self-adjoint
operator $\parallel
M_{f}D_{\vartheta}^{2}M_{f}\parallel_{L^{2}_{\alpha}}=\lambda$
where $\lambda$ is the largest eigenvalue of
$M_{f}D_{\vartheta}^{2}M_{f}$. Let $g\in L^{2}_{\alpha}$ be the
corresponding eigenvector with $\parallel
g\parallel_{L^{2}_{\alpha}}=1$ i.e.
$(M_{f}D_{\vartheta}^{2}M_{f})(g)=\lambda g$. Since $f(z)<1$
$\forall z\in\overline{\mathbb{H}}\setminus\{t\}$ we have
$\parallel (M_{f})h\parallel_{L^{2}_{\alpha}}<\parallel
h\parallel_{L^{2}_{\alpha}}$ $\forall h\in L^{2}_{\alpha}$. Hence
we have
$$\lambda=\parallel\lambda g\parallel_{L^{2}_{\alpha}}=\parallel(M_{f}D_{\vartheta}^{2}M_{f})(g)\parallel_{L^{2}_{\alpha}}<\parallel(D_{\vartheta}^{2}M_{f})(g)\parallel_{L^{2}_{\alpha}}\leq 1.$$
And this implies that
$$\parallel T_{f}D_{\vartheta}\parallel_{A^{2}_{\alpha}}^{2}\leq\lambda<1.$$
Hence under these conditions we have
    \begin{equation*}
        \parallel D_{\vartheta}M_{f}\parallel_{L^{2}_{\alpha}(\mathbb{H})}<
        1\Rightarrow\parallel D_{\vartheta}T_{f}\parallel_{A^{2}_{\alpha}(\mathbb{H})}<
        1\Rightarrow (x,y)\not\in M(\tilde{C}),
    \end{equation*}

    so if $(x,y)\in M(\tilde{C})$, then either $y= \infty$ or $x\in\mathcal{M}_{\infty}$.

    Let $y= \infty$ and $x\in\mathcal{M}$. Let $f\in VMO_{\partial}(\mathbb{H})$ and $\vartheta\in C([0,\infty])$ such that
\begin{equation*}
    0\leq f,\vartheta\leq 1\quad\textrm{and}\quad\hat{T_{f}}(x)=
    \vartheta(y)= 1.
\end{equation*}
Let $\varepsilon>0$ be given and let $t_{0}\in [0,\infty)$ so that
$$1-\varepsilon\leq\vartheta(t)<1\quad\forall t\in [t_{0},\infty).$$
Let $S_{t_{0}}:L^{2}_{\alpha+1}(\mathbb{R}^{+})\rightarrow
L^{2}_{\alpha+1}(\mathbb{R}^{+})$ be defined as
$$S_{t_{0}}f(t)=
\begin{cases}
f(t-t_{0})\quad\textrm{if}\quad t\geq t_{0}\\
 0\quad\textrm{otherwise}
 \end{cases}$$
 Since $\sup_{z\in\mathbb{H}}\{\mid e^{it_{0}z}\mid\}= 1$ $\forall z\in\mathbb{H}$ we
 have $$\parallel T_{e^{it_{0}z}f}\parallel_{A^{2}_{\alpha}}=1$$
 Hence there is $g\in A^{2}_{\alpha}(\mathbb{H})$ so that
 $\parallel g\parallel_{A^{2}_{\alpha}}=1$ and $\parallel
 T_{e^{it_{0}z}f}g\parallel_{A^{2}_{\alpha}}>1-\varepsilon$. Since
 $e^{it_{0}z}\in H^{\infty}(\mathbb{H})$ we have $$T_{e^{it_{0}z}f}=M_{e^{it_{0}z}}T_{f}.$$
 We observe that
 $$\mathcal{F}^{-1}S_{t_{0}}=M_{e^{it_{0}z}}\mathcal{F}^{-1}$$ and
 this implies that
 $$M_{e^{it_{0}z}}=\mathcal{F}S_{t_{0}}\mathcal{F}^{-1}.$$ So we
 have $$\mathcal{F}M_{e^{it_{0}z}}T_{f}=\mathcal{F}(\mathcal{F}^{-1}S_{t_{0}}\mathcal{F})T_{f}=S_{t_{0}}\mathcal{F}T_{f}.$$
 And this implies that
\begin{eqnarray*}
& &\parallel D_{\vartheta}T_{f}(e^{it_{0}z}g)\parallel_{A^{2}_{\alpha}}=\parallel\mathcal{F}^{-1}M_{\vartheta}\mathcal{F}T_{e^{it_{0}z}f}(g)\parallel_{A^{2}_{\alpha}}=\parallel\mathcal{F}^{-1}M_{\vartheta}\mathcal{F}M_{e^{it_{0}z}}T_{f}(g)\parallel_{A^{2}_{\alpha}}\\
& &=\parallel\mathcal{F}^{-1}M_{\vartheta}S_{t_{0}}\mathcal{F}T_{f}(g)\parallel_{A^{2}_{\alpha}}=\parallel M_{\vartheta}S_{t_{0}}\mathcal{F}T_{f}(g)\parallel_{L^{2}_{\alpha+1}}\\
\end{eqnarray*}
$S_{t_{0}}\mathcal{F}T_{f}(g)$ is supported on $[t_{0},\infty)$,
since
$T_{e^{it_{0}z}f}(g)=M_{e^{it_{0}z}}T_{f}(g)=\mathcal{F}^{-1}S_{t_{0}}\mathcal{F}T_{f}$
and $\mathcal{F}^{-1}$ is an isometry we have
$$\parallel S_{t_{0}}\mathcal{F}T_{f}(g)\parallel_{L^{2}_{\alpha+1}}\geq
1-\varepsilon.$$ This implies that
\begin{eqnarray*}
& &\parallel D_{\vartheta}T_{f}(e^{it_{0}z}g)\parallel_{A^{2}_{\alpha}}=\parallel M_{\vartheta}S_{t_{0}}\mathcal{F}T_{f}(g)\parallel_{L^{2}_{\alpha+1}}\\
& &\geq\inf\{\vartheta(t):t>t_{0}\}\parallel
S_{t_{0}}\mathcal{F}T_{f}(g)\parallel_{L^{2}_{\alpha+1}}\geq
(1-\varepsilon)^{2}
\end{eqnarray*}
And since $\parallel e^{it_{0}z}g\parallel_{A^{2}_{\alpha}}\leq 1$
we have
$$\parallel D_{\vartheta}T_{f}\parallel_{A^{2}_{\alpha}(\mathbb{H})}\geq (1-\varepsilon)^{2}$$
for any $\varepsilon>0$. Hence we conclude that
\begin{equation*}
    \parallel
    D_{\vartheta}T_{f}\parallel_{A_{\alpha}^{2}(\mathbb{H})}= 1\Rightarrow (x,\infty)\in
    M(\tilde{C})\quad\forall x\in M(C_{1}).
\end{equation*}
    Now let $x\in\mathcal{M}_{\infty}$ and $y\in [0,\infty]$. Let
    $f\in VMO_{\partial}(\mathbb{H})\cap L^{\infty}(\mathbb{H})$ and $\vartheta\in C([0,\infty])$ such that
\begin{equation*}
    \hat{T_{f}}(x)=\vartheta(y)=1\quad\textrm{and}\quad
    0\leq f,\vartheta\leq 1.
\end{equation*}
Since $VMO_{\partial}(\mathbb{H})\cap L^{\infty}(\mathbb{H})$ is
generated by linear algebraic combinations of functions $f$ and
$\bar{f}$ where $f\in VMO_{\partial}(\mathbb{H})\cap
H^{\infty}(\mathbb{H})$ w.l.o.g one may assume that
$f=f_{0}\bar{f_{0}}$ where $f_{0}\in
VMO_{\partial}(\mathbb{H})\cap H^{\infty}(\mathbb{H})$. By
\cite{Zh} we have $H_{f_{0}}:A^{2}_{\alpha}\rightarrow
(A^{2}_{\alpha})^{\perp}$ is compact where
$$H_{f_{0}}=(I-P)M_{f_{0}}$$ is the Hankel operator with symbol
$f_{0}$. Since $H_{f}=H_{f_{0}}M_{\bar{f_{0}}}$, $H_{f}$ is also
compact.

Since $f\in VMO_{\partial}(\mathbb{H})\cap
L^{\infty}(\mathbb{H})$, for a given $\varepsilon
> 0$ there is a $\delta> 0$ so that
    \begin{equation*}
    \mid\hat{T_{f}}(x)-\frac{1}{\mid Q_{z}\mid}\int_{Q_{z}}f\circ\mathfrak{C}^{-1}(w)dA(w)\mid\leq\varepsilon.
    \end{equation*}
    for all $z\in\mathbb{D}$ with $\mid 1-z\mid<\delta$ where
    $Q_{z}=\{w\in\mathbb{D}:\mid w\mid>\mid z\mid,\mid\arg z-\arg
    w\mid<1-\mid z\mid\}$ and $\mid Q_{z}\mid=(1+\mid
    z\mid)(1-\mid z\mid)^{2}$ is the $dA$ measure of $Q_{z}$.
    Since $\hat{T_{f}}(x)= 1$ and $0\leq f\leq 1$, this implies
    that for all $\varepsilon > 0$ there exists $w_{0} > 0$ such that
    $1-\varepsilon\leq$ $f(w)\leq 1$ for all $w\in\mathbb{H}$ with $\mid
    w\mid > w_{0}$. We have $M_{f}=T_{f}+H_{f}$ which is valid in
    much more general contexts. Since $H_{f}$ is compact for any
    $g\in A_{\alpha}^{2}(\mathbb{H})$, we have
    $$\lim_{w\rightarrow\infty}\parallel H_{f}(S_{w}g)\parallel_{L^{2}_{\alpha}(\mathbb{H})}=0$$
    where $S_{w}g(z)=g(z-w)$ ($w\in\mathbb{R}$) is the translation operator, since
    $S_{w}g$ converges to $0$ weakly as $w$ tends to infinity and
    $H_{f}$ is compact.
  Let $g\in A^{2}_{\alpha}(\mathbb{H})$ so that $\parallel g\parallel_{A^{2}_{\alpha}}=1$ and $\parallel
  D_{\vartheta}g\parallel_{A^{2}_{\alpha}}>1-\frac{\varepsilon}{2}$. Let $K\subset\mathbb{H}$ be such that
  $\overline{K}\subset\mathbb{C}$ is compact and
  $$(\int_{K}\mid D_{\vartheta}g(z)\mid^{2} dA_{\alpha}(z))^{\frac{1}{2}}>1-\varepsilon.$$
  Let $w>w_{0}$ so that
  $$\parallel H_{f}(S_{w}D_{\vartheta}g)\parallel_{L^{2}_{\alpha}(\mathbb{H})}\leq\varepsilon.$$
  Then we have
\begin{eqnarray*}
& &\parallel T_{f}(S_{w}D_{\vartheta}g)\parallel_{A_{\alpha}^{2}(\mathbb{H})}=\parallel(M_{f}-H_{f})(S_{w}D_{\vartheta}g)\parallel_{A_{\alpha}^{2}(\mathbb{H})}\\
& &\geq\parallel M_{f}(S_{w}D_{\vartheta}g)\parallel_{L_{\alpha}^{2}(\mathbb{H})}-\parallel H_{f}(S_{w}D_{\vartheta}g)\parallel_{L_{\alpha}^{2}(\mathbb{H})}\geq\parallel M_{f}(S_{w}D_{\vartheta}g)\parallel_{L_{\alpha}^{2}(\mathbb{H})}-\varepsilon\\
\end{eqnarray*}
Since $f(z)>1-\varepsilon$ $\forall\mid z\mid>w_{0}$ we have
\begin{eqnarray*}
& &\parallel M_{f}S_{w}D_{\vartheta}g\parallel_{L^{2}_{\alpha}(\mathbb{H})}\geq(\int_{w+K}\mid f(z)D_{\vartheta}(g)(z-w)\mid^{2}dA_{\alpha}(z))^{\frac{1}{2}}=\\
& &(\int_{K}\mid f(u+w)D_{\vartheta}(g)(u)\mid^{2}dA_{\alpha}(u))^{\frac{1}{2}}\\
& &\geq\inf\{f(z):z\in w+K\}(\int_{K}\mid
D_{\vartheta}(g)(z)\mid^{2}dA_{\alpha}(z))^{\frac{1}{2}}\geq(1-\varepsilon)^{2}
\end{eqnarray*}
So we have
$$\parallel T_{f}S_{w}D_{\vartheta}g\parallel_{A^{2}_{\alpha}}=\parallel T_{f}D_{\vartheta}S_{w}g\parallel_{A^{2}_{\alpha}}\geq (1-\varepsilon)^{2}-\varepsilon$$
since $D_{\vartheta}S_{w}=S_{w}D_{\vartheta}$ $\forall
w\in\mathbb{R}$($S_{w}=D_{e^{2\pi iwt}}$). Since $S_{w}$ is
unitary for all $w\in\mathbb{R}$, we have $\parallel
S_{w}g\parallel_{A^{2}_{\alpha}}=1$ and this implies that
$$\parallel T_{f}D_{\vartheta}\parallel_{A^{2}_{\alpha}}\geq(1-\varepsilon)^{2}-\varepsilon$$
for any $\varepsilon>0$.
 Hence we have
$$\parallel T_{f}D_{\vartheta}\parallel_{A^{2}_{\alpha}}=1$$
and $(x,y)\in M(\tilde{C})$ for all $x\in\mathcal{M}_{\infty}$.
\end{proof}

\section{main results}

In this section we characterize the essential spectra of
quasi-parabolic composition operators with translation functions
in $VMO_{\partial}$ class which is the main aim of the paper. In
doing this we will heavily use Banach algebraic methods.

We will need the following result whose proof uses a theorem of
\cite{Hof}(p. 171) and which gives the values $\hat{T_{f}}(x)$ of
$T_{f}$ for $x\in\mathcal{M}_{\infty}$ on the fibers of
$\mathcal{M}$ the maximal ideal space of
$\mathcal{T}(VMO_{\partial})/K(A^{2}_{\alpha}(\mathbb{H}))$ at
infinity:
\begin{proposition}
Let $\psi\in VMO_{\partial}\cap H^{\infty}(\mathbb{H})$ and
$\mathcal{M}$ be the maximal ideal space of
$\mathcal{T}(VMO_{\partial})/K(A^{2}_{\alpha}(\mathbb{H}))$. Let
$\mathcal{M}_{\infty}$ be the fiber of $\mathcal{M}$ at infinity
which is defined as
$$\mathcal{M}_{\infty}=\{x\in\mathcal{M}:x|_{\mathcal{T}(C(\dot{\overline{\mathbb{H}}}))/K(A^{2}_{\alpha})}=\delta_{\infty}\}$$
where $\delta_{\infty}([T_{f}])=f(\infty)$. Then we have
$$\{\hat{[T_{\psi}]}(x):x\in\mathcal{M}_{\infty}\}=\mathcal{R}_{\infty}(\psi^{\ast})$$
where $\psi^{\ast}\in L^{\infty}(\mathbb{R})$ is the boundary
value function of $\psi$ i.e. $\psi^{\ast}(x)=\lim_{y\rightarrow
0}\psi(x+iy)$ and
$$\mathcal{R}_{\infty}(\psi^{\ast})=\{\zeta\in\mathbb{C}:\mid\{x:\mid\psi^{\ast}(x)-\zeta\mid\leq\varepsilon\}\cap(\mathbb{R}\setminus [-n,n])\mid>0\forall\varepsilon>0,\forall n\in\mathbb{N}\} $$
is the local essential range of $\psi^{\ast}$ at infinity.
\end{proposition}
\begin{proof}
Let $A=C^{\ast}(\{T_{\psi}\}\cup\{T_{f}:f\in
C(\dot{\overline{\mathbb{H}}})\})/K(A^{2}_{\alpha}(\mathbb{H}))\subset\mathcal{B}(A^{2}_{\alpha}(\mathbb{H}))/K(A^{2}_{\alpha}(\mathbb{H}))$.
$A$ is also a C*-subalgebra of
$\Psi(VMO_{\partial},C([0,\infty]))/K(A^{2}_{\alpha}(\mathbb{H}))$.
We observe that $A$ is isometrically isomorphic to $\tilde{A}$
where $\tilde{A}=C^{\ast}(\{\psi^{\ast}\}\cup
C(\dot{\mathbb{R}}))\subset L^{\infty}(\mathbb{R})$, via the
correspondence $f$ $\rightarrow$ $[T_{f}]$. Since the ideal
generated by a maximal ideal $I\subset A$ in
$\Psi=\Psi(VMO_{\partial},C([0,\infty]))/K(A^{2}_{\alpha}(\mathbb{H}))$
is contained in a maximal ideal of $\Psi$, we have
$$\{\hat{[T_{\psi}]}(x):x\in\mathcal{M}_{\infty}\}=\{\hat{[T_{\psi}]}(x):x\in
M(A)_{\infty}\}$$ where $M(A)_{\infty}=\{x\in
M(A):x|_{\mathcal{T}(C(\dot{\overline{\mathbb{H}}}))/K(A^{2}_{\alpha}(\mathbb{H}))}=\delta_{\infty}\}$
with $\delta_{\infty}([T_{f}])=f(\infty)$. Since $A$ is
isometrically isomorphic to $\tilde{A}$ we have
$$\{\hat{[T_{\psi}]}(x):x\in M(A)_{\infty}\}=\{\hat{\psi^{\ast}}(x):x\in M(\tilde{A})_{\infty}\}$$
Similarly since the ideal generated by a
maximal ideal $I\subset\tilde{A}$ in $L^{\infty}(\mathbb{R})$ is
contained in a maximal ideal in $L^{\infty}(\mathbb{R})$, we have
$$\{\hat{\psi^{\ast}}(x):x\in M(\tilde{A})_{\infty}\}=\{\hat{\psi^{\ast}}(x):x\in M(L^{\infty}(\mathbb{R}))_{\infty}\}$$
By the theorem of \cite{Hof}(p. 171) we have
$$\{\hat{\psi^{\ast}}(x):x\in M(L^{\infty}(\mathbb{R}))_{\infty}\}=\mathcal{R}_{\infty}(\psi^{\ast})$$
Therefore we have
$$\{\hat{[T_{\psi}]}(x):x\in\mathcal{M}_{\infty}\}=\mathcal{R}_{\infty}(\psi^{\ast})$$
\end{proof}

Firstly we have the following result on the upper half-plane:

\begin{thma}
Let $\psi\in$ $VMO_{\partial}(\mathbb{H})\cap
H^{\infty}(\mathbb{H})$ such that $\Im(\psi(z)) > \epsilon > 0$
for all $z\in$ $\mathbb{H}$ then for $\varphi(z)=$ $z+\psi(z)$ and
$\alpha>-1$ we have
\begin{itemize}
\item (i)\quad $C_{\varphi}:$ $A^{2}_{\alpha}(\mathbb{H})\rightarrow$
$A^{2}_{\alpha}(\mathbb{H})$ is essentially normal \\
\item (ii)\quad$\sigma_{e}(C_{\varphi})=\{e^{izt}:t\in [0,\infty],
z\in\mathcal{R}_{\infty}(\psi^{\ast})\}\cup\{0\}$
\end{itemize}
where\quad $\mathcal{R}_{\infty}(\psi^{\ast})$ is the local
essential range of $\psi^{\ast}\in L^{\infty}(\mathbb{R})$ at
$\infty$ and $\psi^{\ast}$ is the boundary limit value function of
$\psi$.
\end{thma}

\begin{proof}
By Proposition 3 we have the following series expansion for
$C_{\varphi}$:
\begin{equation}
C_{\varphi} =
    \sum_{n=0}^{\infty}\frac{\Gamma(n+2+\alpha)}{n!\Gamma(\alpha+2)}T_{\tau^{n}}D_{\phi_{n}}
\end{equation}
where $\tau(z)=$ $\psi(z)-i\beta$ and $\phi_{n}(t)=$ $\frac{(2\pi
it)^{n}e^{-2\pi\beta t}}{(\alpha+2)(\alpha+3)...(\alpha+n+1)}$. So
we conclude that if $\psi\in$ $VMO_{\partial}\cap
H^{\infty}(\mathbb{H})$ with $\Im({\psi(z)})>$ $\epsilon
> 0$ then
\begin{equation*}
C_{\varphi}\in\Psi(VMO_{\partial},C([0,\infty]))
\end{equation*}
where $\varphi(z)=$ $z+\psi(z)$. Since
$\Psi(VMO_{\partial},C([0,\infty]))/K(A^{2}_{\alpha}(\mathbb{H}))$
is commutative, for any $T\in\Psi(VMO_{\partial},C([0,\infty]))$
we have $T^{*}\in \Psi(VMO_{\partial},C([0,\infty]))$ and
\begin{equation}
[TT^{*}]=[T][T^{*}]=[T^{*}][T]=[T^{*}T].
\end{equation}
This implies that $(TT^{*}-T^{*}T)\in
K(A^{2}_{\alpha}(\mathbb{H}))$. Since
$C_{\varphi}\in\Psi(VMO_{\partial},C([0,\infty]))$ we also have
\begin{equation*}
(C_{\varphi}^{*}C_{\varphi}-C_{\varphi}C_{\varphi}^{*})\in
K(A^{2}_{\alpha}(\mathbb{H})).
\end{equation*}
This proves (i).

For (ii) we look at the values of $\Gamma[C_{\varphi}]$ at
$M(\Psi)=M(\Psi(VMO_{\partial},C([0,\infty]))/K(A^{2}_{\alpha}(\mathbb{H})))$
where $\Gamma$ is the Gelfand transform of
$\Psi(VMO_{\partial},C([0,\infty]))/K(A^{2}_{\alpha}(\mathbb{H}))$.
By Theorem 6 we have
\begin{equation*}
M(\Psi)=(\mathcal{M}\times\{\infty\})\cup(\mathcal{M}_{\infty}\times
[0,\infty])
\end{equation*}
where
$\mathcal{M}=M(\mathcal{T}(VMO_{\partial})/K(A^{2}_{\alpha}(\mathbb{H}))$.
By equation (6) we have the Gelfand transform
$\Gamma[C_{\varphi}]$ of $C_{\varphi}$ at $t=$ $\infty$ as
\begin{equation}
(\Gamma[C_{\varphi}])(x,\infty)=\sum_{j=0}^{\infty}\frac{\Gamma(n+2+\alpha)}{j!\Gamma(\alpha+2)}(\hat{T_{\tau}}(x))^{j}\phi_{j}(\infty)=0\quad\forall
x\in\mathcal{M}
\end{equation}
since $\phi_{j}(\infty)=0$ for all $j\in\mathbb{N}$ where
$\phi_{j}(t)=\frac{(2\pi it)^{j}e^{-2\pi\beta
t}}{(\alpha+2)(\alpha+3)...(\alpha+n+1)}$. We calculate
$\Gamma[C_{\varphi}]$ of $C_{\varphi}$ for $x\in$
$\mathcal{M}_{\infty}$ as
\begin{eqnarray}
& &(\Gamma[C_{\varphi}])(x,t)=\\
 &\nonumber &\left(\Gamma\left[\sum_{j=0}^{\infty}\frac{\Gamma(n+2+\alpha)}{j!\Gamma(\alpha+2)}([T_{\tau}])^{j}D_{\frac{(2\pi
 it)^{j}e^{-2\pi\beta t}}{(\alpha+2)(\alpha+3)...(\alpha+n+1)}}\right]\right)(x,t)=\\
&\nonumber &\sum_{j=0}^{\infty}\frac{\Gamma(n+2+\alpha)}{j!\Gamma(\alpha+2)}\hat{[T_{\tau}]}(x)^{j}\frac{(2\pi it)^{j}e^{-2\pi\beta t}}{(\alpha+2)(\alpha+3)...(\alpha+n+1)}\\
&\nonumber
&=\sum_{j=0}^{\infty}\frac{1}{j!}\hat{[T_{\tau}]}(x)^{j}(2\pi
it)^{j}e^{-2\pi\beta t}=e^{2\pi i\hat{[T_{\psi}]}(x)t}
\end{eqnarray}
 for all
$x\in$ $\mathcal{M}_{\infty}$ and $t\in$ $[0,\infty]$ since
$$(\alpha+2)(\alpha+3)...(\alpha+n+1)=\frac{\Gamma(n+2+\alpha)}{\Gamma(\alpha+2)}.$$
So we have
$\Gamma[C_{\varphi}]$ as the following:

\begin{eqnarray}
& &\Gamma([C_{\varphi}])(x,t)=
\begin{cases}
e^{2\pi i\hat{[T_{\psi}]}(x)t}\quad\textrm{if}\quad x\in\mathcal{M}_{\infty} \\
0\quad\qquad\textrm{if}\quad t=\infty
\end{cases}
\end{eqnarray}
Since
$\Psi=\Psi(VMO_{\partial},C([0,\infty]))/K(A^{2}_{\alpha}(\mathbb{H}))$
is a commutative Banach algebra with identity, by equations (2)
and (10) we have
\begin{eqnarray}
\sigma_{\Psi}([C_{\varphi}])=\{\Gamma[C_{\varphi}](x,t):(x,t)\in
M(\Psi(VMO_{\partial},C([0,\infty]))/K(A^{2}_{\alpha}(\mathbb{H}))\}=\\
\{e^{2\pi i x([T_{\psi}])t}:x\in\mathcal{M}_{\infty},t\in\nonumber
[0,\infty)\}\cup\{0\}
\end{eqnarray}
Since $\Psi$ is a closed *-subalgebra of the Calkin algebra
$\mathcal{B}(A^{2}_{\alpha}(\mathbb{H}))/K(A^{2}_{\alpha}(\mathbb{H}))$
which is also a C*-algebra, by equation (1) we have
\begin{equation}
\sigma_{\Psi}([C_{\varphi}])=\sigma_{\mathcal{B}(A^{2}_{\alpha})/K(A^{2}_{\alpha})}([C_{\varphi}]).
\end{equation}
But by definition
$\sigma_{\mathcal{B}(A^{2}_{\alpha})/K(A^{2}_{\alpha})}([C_{\varphi}])$
is the essential spectrum of $C_{\varphi}$. Hence we have
\begin{equation}
\sigma_{e}(C_{\varphi})=\{e^{i\hat{[T_{\psi}]}(x)t}:x\in\mathcal{M}_{\infty},t\in
[0,\infty)\}\cup\{0\}.
\end{equation}

Now it only remains for us to understand what the set
$\{\hat{[T_{\psi}]}(x)=x([T_{\psi}]):x\in\mathcal{M}_{\infty}\}$
looks like, where $\mathcal{M}_{\infty}$ is as defined in Theorem
6. By Proposition 7 we have
\begin{equation*}
\{\hat{[T_{\psi}]}(x):x\in\mathcal{M}_{\infty}\}=\mathcal{R}_{\infty}(\psi^{\ast}).
\end{equation*}
By Proposition 7 and equation (13) we have
\begin{eqnarray*}
& &\sigma_{e}(C_{\varphi})=\{(\Gamma[C_{\varphi}])(x,t):(x,t)\in
M(\Psi(VMO_{\partial},C([0,\infty]))/K(A^{2}_{\alpha}(\mathbb{H})))\}=\\
& &\{e^{izt}:t\in [0,\infty),
z\in\mathcal{R}_{\infty}(\psi^{\ast})\}\cup\{0\}
\end{eqnarray*}
\end{proof}
The local essential range $\mathcal{R}_{1}(f)$ of $f\in
L^{\infty}(\mathbb{T})$ at $1$ is defined to be the set of points
  $\zeta\in$ $\mathbb{C}$ for which the set $\{z\in\mathbb{T}:\mid f(z)-\zeta\mid<\varepsilon\}\cap
  S_{1,r}$ has positive Lebesgue measure $\forall\varepsilon>0$ and $\forall r>0$ where $S_{1,r}=\{z\in\mathbb{T}:\mid
  z-1\mid<r\}$. Conjugating by Cayley transform we observe that
  $$\mathcal{R}_{1}(f)=\mathcal{R}_{\infty}(f\circ\mathfrak{C}).$$
  Using this observation we prove the analogous result for the
  unit disc case:
\begin{thmb}
 Let $\varphi:\mathbb{D}\rightarrow$ $\mathbb{D}$ be an
analytic self-map of $\mathbb{D}$ such that
\begin{equation*}
\varphi(z)=\frac{2iz+\eta(z)(1-z)}{2i+\eta(z)(1-z)}
\end{equation*}
where $\eta\in$ $H^{\infty}(\mathbb{D})$ with $\Im(\eta(z)) >
\epsilon> 0$ for all $z\in$ $\mathbb{D}$. If $\eta\in$
$VMO_{\partial}(\mathbb{D})\cap H^{\infty}$ then we have
\begin{itemize}
\item (i)\quad $C_{\varphi}:$ $A^{2}_{\alpha}(\mathbb{D})\rightarrow$
$A^{2}_{\alpha}(\mathbb{D})$ is essentially normal
\item (ii)\quad$\sigma_{e}(C_{\varphi})=\{e^{izt}:t\in [0,\infty],
z\in\mathcal{R}_{1}(\eta^{\ast})\}\cup\{0\}$
\end{itemize}
where\quad $\mathcal{R}_{1}(\eta^{\ast})$ is the local essential
range of $\eta^{\ast}\in L^{\infty}(\mathbb{T})$ at $1$ and
$\eta^{\ast}$ is the boundary limit value function of $\eta$.
\end{thmb}

\begin{proof}
Using the isometric isomorphism
$\Phi:A^{2}_{\alpha}(\mathbb{D})\longrightarrow$
$A^{2}_{\alpha}(\mathbb{H})$ introduced in section 2, if
$\varphi:\mathbb{D}\rightarrow$ $\mathbb{D}$ is of the form
\begin{equation*}
\varphi(z)=\frac{2iz+\eta(z)(1-z)}{2i+\eta(z)(1-z)}
\end{equation*}
where $\eta\in$ $H^{\infty}(\mathbb{D})$ satisfies $\Im(\eta(z))>$
$\delta
> 0$ then, by equation (9), for $\tilde{\varphi}=$
$\mathfrak{C}^{-1}\circ\varphi\circ\mathfrak{C}$ we have
$\tilde{\varphi}(z)=$ $z+\eta\circ\mathfrak{C}(z)$ and
\begin{equation}
\Phi\circ C_{\varphi}\circ\Phi^{-1}=
T_{(\frac{z+i+\eta\circ\mathfrak{C}(z)}{z+i})^{2}}C_{\tilde{\varphi}}.
\end{equation}
For $\eta\in$ $VMO_{\partial}\cap H^{\infty}(\mathbb{D})$ we have
both
\begin{equation*}
C_{\tilde{\varphi}}\in\Psi(VMO_{\partial},C([0,\infty]))\quad\textrm{and}\quad
T_{(\frac{z+i+\eta\circ\mathfrak{C}(z)}{z+i})^{2}}\in\Psi(VMO_{\partial},C([0,\infty]))
\end{equation*}
and hence
\begin{equation*}
 \Phi\circ
C_{\varphi}\circ\Phi^{-1}\in\Psi(VMO_{\partial},C([0,\infty])).
\end{equation*}
Since
$\Psi(VMO_{\partial},C([0,\infty]))/K(A^{2}_{\alpha}(\mathbb{H}))$
is commutative and $\Phi$ is an isometric isomorphism, (i) follows
from the argument $C_{\varphi}$ is essentially normal if and only
if $\Phi\circ C_{\varphi}\circ\Phi^{-1}$ is essentially normal and
by equation (14).

For (ii) we look at the values of $\Gamma[\Phi\circ
C_{\varphi}\circ\Phi^{-1}]$ at
$M(\Psi(VMO_{\partial},C([0,\infty]))/K(A^{2}_{\alpha}))$ where
$\Gamma$ is the Gelfand transform of
$\Psi(VMO_{\partial},C([0,\infty]))/K(A^{2}_{\alpha})$. Again
applying the Gelfand transform for
\begin{equation*}
(x,\infty)\in
M(VMO_{\partial},C([0,\infty]))/K(A^{2}_{\alpha}))\subset\mathcal{M}\times
[0,\infty]
\end{equation*}
 we have
\begin{equation*}
(\Gamma[\Phi\circ
C_{\varphi}\circ\Phi^{-1}])(x,\infty)=((\Gamma[T_{(\frac{z+i+\eta\circ\mathfrak{C}(z)}{z+i})^{2}}])(x,\infty))((\Gamma[C_{\tilde{\varphi}}])(x,\infty))
\end{equation*}

Appealing to equation (8) we have
$(\Gamma[C_{\tilde{\varphi}}])(x,\infty)= 0$ for all $x\in$
$\mathcal{M}$ hence we have
\begin{equation*}
(\Gamma[\Phi\circ C_{\varphi}\circ\Phi^{-1}])(x,\infty)=0
\end{equation*}
 for all $x\in$ $\mathcal{M}$. Applying the Gelfand transform for
$$(x,t)\in\mathcal{M}_{\infty}\times [0,\infty]\subset
M(\Psi(VMO_{\partial},C([0,\infty]))/K(A^{2}_{\alpha})$$ we have
\begin{eqnarray*}
& &(\Gamma[\Phi\circ C_{\varphi}\circ\Phi^{-1}])(x,t)=((\Gamma[T_{(\frac{z+i+\eta\circ\mathfrak{C}(z)}{z+i})^{2}}])(x,t))((\Gamma[C_{\tilde{\varphi}}])(x,t))\\
&
&=((1+\Gamma[T_{\eta\circ\mathfrak{C}}](x,t)\Gamma[T_{(\frac{1}{z+i})}](x,t))^{2}((\Gamma[C_{\tilde{\varphi}}])(x,t))
\end{eqnarray*}
Since $x\in$ $\mathcal{M}_{\infty}$ we have
\begin{equation*}
(\Gamma[T_{\frac{1}{i+z}}])(x,t)=\hat{T_{(\frac{1}{i+z})}}(x)=x(T_{\frac{1}{i+z}})=0.
\end{equation*}
Hence we have
\begin{equation*}
(\Gamma[\Phi\circ
C_{\varphi}\circ\Phi^{-1}])(x,t)=(\Gamma[C_{\tilde{\varphi}}])(x,t)
\end{equation*}
for all $(x,t)\in$ $\mathcal{M}_{\infty}\times [0,\infty]$.
Moreover we have
\begin{equation*}
(\Gamma[\Phi\circ
C_{\varphi}\circ\Phi^{-1}])(x,t)=(\Gamma[C_{\tilde{\varphi}}])(x,t)
\end{equation*}
for all $(x,t)\in
M(\Psi(VMO_{\partial},C([0,\infty]))/K(A^{2}_{\alpha}))$.
Therefore by similar arguments in Theorem A (equations (11) and
(12)) we have
\begin{equation*}
\sigma_{e}(\Phi\circ
C_{\varphi}\circ\Phi^{-1})=\sigma_{e}(C_{\tilde{\varphi}}).
\end{equation*}
 By Theorem A (together with equation (14)) we have
\begin{equation*}
\sigma_{e}(C_{\varphi})=\sigma_{e}(C_{\tilde{\varphi}})=\{e^{izt}:z\in\mathcal{R}_{\infty}(\eta\circ\mathfrak{C}^{\ast})=\mathcal{R}_{1}(\eta^{\ast}),t\in
[0,\infty]\}.
\end{equation*}
\end{proof}
\bibliographystyle{amsplain}
\addtocontents{toc}{\protect\contentsline {part}{}{}}

\end{document}